\documentclass[12pt]{article}
\usepackage{amssymb,url,color}
\usepackage[hypertex]{hyperref}
\usepackage{amsmath, amsthm, amsfonts, amssymb}
\usepackage{txfonts}

\newtheorem{thm}{Theorem}[section]
\newtheorem{cor}[thm]{Corollary}
\newtheorem{lem}[thm]{Lemma}
\newtheorem{prop}[thm]{Proposition}

\newtheorem{remarks}[thm]{Remarks}

\theoremstyle{definition}
\newtheorem{defn}{Definition}[section]

\numberwithin{equation}{section} \theoremstyle{remark}

\title{\bf Strassen's invariance principle for random walk in random environment}
\author{
\bf Guangyu Yang,\thanks{Department of Mathematics, Zhengzhou
University, 450052 Henan, China.
\newline
E-mail:study\_yang@yahoo.com.cn}\ \ \ \ \bf Yu Miao\thanks{ College
of Mathematics and Information Science, Henan Normal University,
453007 Henan, China.
\newline
E-mail:yumiao728@yahoo.com.cn} ~~and \ \ \bf Dihe Hu\thanks{School
of Mathematics and Statistics, Wuhan University, 430072 Hubei,
China.
\newline E-mail:dhhu@whu.edu.cn}}
\date{}

\newcommand{\ee}{\mathbb{E}}

\newcommand{\nn}{\mathbb{N}}
\newcommand{\rr}{\mathbb{R}}
\newcommand{\pp}{\mathbb{P}}

\newcommand{\zz}{\mathbb{Z}}

\def\<{\langle}
\def\>{\rangle}

\def\beq{\begin{equation}}
\def\deq{\end{equation}}

\def\bdef{\begin{defn}}
\def\ndef{\end{defn}}

\def\bthm{\begin{thm}}
\def\nthm{\end{thm}}

\def\bprop{\begin{prop}}
\def\nprop{\end{prop}}

\def\brmk{\begin{remarks}}
\def\nrmk{\end{remarks}}

\def\bexa{\begin{exa}}
\def\nexa{\end{exa}}

\def\blem{\begin{lem}}
\def\nlem{\end{lem}}

\def\bcor{\begin{cor}}
\def\ncor{\end{cor}}

\def\bexe{\begin{exe}}
\def\nexe{\end{exe}}

\def\bprf{\begin{proof}}
\def\nprf{\end{proof}}

\def\bdes{\begin{description}}
\def\ndes{\end{description}}

\begin{document}
\maketitle

\begin{abstract}
In this paper, we consider random walk in random environment on
$\mathbb{Z}^{d}\,(d\geq1)$ and prove the Strassen's strong
invariance principle for this model, via martingale argument and the
theory of fractional coboundaries of Derriennic and Lin \cite{DL},
under some conditions which require the variance of the quenched
mean has a subdiffusive bound. The results partially fill the gaps
between law of large numbers and central limit theorems.
\end{abstract}

\textbf{Keywords}: random walk in random environment, Strassen's
invariance principle, the law of iterated logarithm, fractional
coboundaries.

\vskip10pt \textbf{2000 MR Subject Classification : }  60B10, 60F15.

\section{Introduction}\label{sec1}

Random motions in random media gather a variety of probability
models often originated from physical science, such as solid
physics, biophysics and so on. Random walk in a random environment
is one of the basic models. Various interesting problems arise when
we consider the different possible limit theorems, such as the $0-1$
law, law of large numbers, central limit theorems, large deviations
and so on. See also the lecture notes given by Sznitman \cite{Sza},
Molchanov \cite{Mo} and Zeitouni \cite{Ze} for a survey. The main
object of this work is to prove the invariance principle for the law
of iterated logarithm for a class of random walks in certain random
environments. In this model, an environment is a collection of
transition probabilities $\omega=(\pi_{x,y})_{x,y\in\mathbb{Z}^{d}}
\in\wp^{\mathbb{Z}^{d}}$, where
$\wp=\{(p_{z})_{z\in\mathbb{Z}^d},p_{z}\geq0,\sum p_{z}=1\}$ a
family of distributions on $\mathbb{Z}^d$. We denote the space of
all such transition probabilities by $\Omega$. The space $\Omega$ is
endowed with the canonical product $\sigma$-algebra $\Im$. On the
space of environments $(\Omega,\Im)$, we are given a T-invariant
probability $\pp$ with $(\Omega,\Im,(T_{z})_{z\in\mathbb{Z}^d},\pp)$
ergodic, where $\{T_{z},z\in\mathbb{Z}^d\}$ denote the canonical
shift on $\Omega$, i.e.
$\pi_{x,y}(T_{z}\omega)=\pi_{x+z,y+z}(\omega)$. The environments are
called independent identical distribution (in short i.i.d.), if the
family of random probabilities vectors $\{(\pi_{x,y})_{
y\in\mathbb{Z}^d},x\in\mathbb{Z}^d\}$ are i.i.d..

We now turn to describe the walk related to the environment. First,
an environment $\omega$ is chosen from the distribution $\pp$, and
kept fixed throughout the time evolution. Then the random walk in
environment $\omega$, is a canonical Makov chain $X:=(X_{n},n\geq0)$
on $(\mathbb{Z}^d)^{\mathbb{N}}$, with state space $\mathbb{Z}^d$
and law $P^\omega_{z}$, under which
\begin{align}
P^\omega_{z}(X_{0}=z)=1,\ \ \ \ P^\omega_{z}(X_{n+1}=y |
X_{n}=x)=\pi_{x,y}(\omega).
\end{align}
The law $P^\omega_{z}$ is called the quenched law. Then we can also
define a measure in the sense of averaging the environments,
\begin{align}
P_{z}=\int P^\omega_{z}d\pp.
\end{align}
The law $P_{z}$ is called the annealed law. Obviously, under
$P_{z}$, $(X_{n},n\geq0)$ is not a Markov chain in general. We also
use $\ee,\,E_{z},\,E^\omega_{z}$ for the expectations operators with
respect to $\pp,\,P_{z},\;{\rm and} \;P^\omega_{z}$ respectively.

Our goal in this paper, is to consider the invariance principle for
the law of iterated logarithm for the random walk in environment,
under some assumptions introduced in Section \ref{sec2}. The results
of iterated logarithm types for Sinai's random walk in random
environment can be found in Hu and Shi \cite{HS}.

It is well known that the law of iterated logarithm (in short LIL)
is closely related to the central limit theorems (in short CLT) in
some sense. By the technique of split chains and regeneration, Chen
\cite{Cha} systematically studied the CLT and LIL for ergodic Markov
chain under the frame of Harris recurrent.  Bhattacharya \cite{Bh}
gave the functional CLT and LIL for Markov processes. And Kifer
\cite{Ki} obtained the CLT and LIL for Markov chain in random
environment via certain mixing assumptions and via the martingale
approach.

Note that, Rassoul-Agha and Sepp\"{a}l\"{a}inen \cite{RSb} mainly
rely on the invariance principle for vector-valued martingale, so it
is possible to obtain the invariance principle for LIL for random
walk in random environment under suitable conditions, only if we can
develop the corresponding theory for vector-valued martingale. In
the case of real-valued martingale, the Skorokhod representation
plays an important role, for example, Hall and Heyde \cite{HH}.
However, we encounter the essential difficulties, when considering
the vector-valued martingale, since Monrad and Philipp \cite{MoP}
proved that it is impossible to embed a general $\rr^d$-valued
martingale in an $\rr^d$-valued Gassian process.

In the present paper, we will use essentially the strategy of
Maxwell and Woodroofe \cite{MW} and the method developed by Morrow
and Philipp \cite{MP} and Zhang \cite{Zh}. Moreover, we identify the
$\lim\sup$ in LIL just the square root of the trace of the diffusion
matrix corresponding to functional CLT. And this partially fills the
gaps between law of large numbers and central limit theorems.


\section{Some preliminaries and main results}\label{sec2}

In this section, we will give some assumptions and state our main
results. Let us start with the construction of the auxiliary Markov
chain.

For any $\omega\in\Omega$, let
$\bar{\omega}:=(\bar{\omega}(n)=T_{X_{n}}\omega,n\geq0)$, then
$\bar{\omega}$ is a Markov chain on $\Omega$ with transition
operator
\begin{align}
\Pi f(\omega)=\sum_{x\in\zz^d}\pi_{0,x}(\omega)f(T_{x}\omega),
\end{align}
where $f$ is a bounded measurable function defined on $\Omega$, and
with the one step transition kernel,
\begin{align}
q(\omega,A)=P_0^{\omega}(T_{X_1}\omega\in A),\;\;A\in\Im.
\end{align}

In this paper, we always assume that there exists a probability
measure $\pp_\infty$ on the measurable space $(\Omega,\Im)$ that is
invariant for the transition $\Pi$ and ergodic for the Markov
process with generator $\Pi-I$. Then, the operator $\Pi$ can be
extended to a contraction on $L^p(\pp_\infty)$, for every
$p\in[1,\infty]$. When the initial distribution is $\pp_\infty$, we
will denote this Markov process by $\hat{P}_{0}^{\infty}$. Let
$P_{0}^\infty:=\int P_{0}^{\omega}d\pp_\infty$, and
$\mathbb{E}_\infty, \,E_{0}^\infty$ the corresponding expectation
operators. Note that $\hat{P}_{0}^{\infty}$ is the probability
measure induced by $P_{0}^\infty$ and $(T_{X_n}\omega)$ onto
$\Omega^{\nn}$. With these notations, the measure
\begin{align}
\nu^\infty(d\omega_0,d\omega_1):=q(\omega_0,d\omega_1)\pp_\infty(d\omega_0)
\end{align}
describes the law of $(\omega,\,T_{X_1}\omega)$ under $P_0^\infty$.

Next, we consider the asymptotic Poisson's equation. For any
$\epsilon>0$, let $h_{\epsilon}$ be the solution to the equation
\begin{align}\label{possion equation}
(1+\epsilon)h-\Pi h=g,
\end{align}
where $g$ is a function defined on $\Omega$ such that $\int g
d\pp_\infty=0$ and $\int g^2 d\pp_\infty<\infty$. In fact,
\begin{align}
h_{\epsilon}=\sum_{k=1}^{\infty}(1+\epsilon)^{-k}\Pi^{k-1}g\in
L^2(\pp_\infty)
\end{align}
is the solution of the equation (\ref{possion equation}). We also
define
\begin{align}
S_{n}(g):=\sum_{k=0}^{n-1}g(T_{X_{k}}\omega)
\end{align}
and
\begin{align}
H_{\epsilon}(\omega_{0},\omega_{1}):=h_{\epsilon}(\omega_1)-\Pi
h_{\epsilon}(\omega_0).
\end{align}
Then we have
\begin{align}
S_{n}(g)&:=\sum_{k=0}^{n-1}g(T_{X_{k}}\omega)\nonumber\\
&=\sum_{k=0}^{n-1}\{(1+\epsilon)h_{\epsilon}
(T_{X_{k}}\omega)-\Pi h_{\epsilon} (T_{X_{k}}\omega)\}\nonumber\\
&= M^{\epsilon}_{n}+R^{\epsilon}_{n}+\epsilon S_{n}(h_{\epsilon}),
\end{align}
where
$M^{\epsilon}_{n}=\sum_{k=0}^{n-1}H_{\epsilon}(T_{X_{k}}\omega,T_{X_{k+1}}\omega)$,
$R^{\epsilon}_{n}=h_{\epsilon}(\omega)-h_{\epsilon}(T_{X_{n}}\omega)$.

In order to discuss the Poisson's equation ulteriorly, we need
introduce some assumptions.
\vspace{3mm}\\
\textbf{Assumptions}\vspace{3mm}

\noindent(A1) {\it There exists a constant $M<\infty$ such that}
\begin{align}
\pp_\infty(\pi_{x,y}(\omega)=0, \rm{when}\; |x-y|>M)=1),
\end{align}
where $|\cdot|$ denotes the Euclidean distance.
\vspace{3mm}

\noindent(A2) {\it There exists an $\alpha<1/2$ such that}
\begin{align}
\sqrt{\ee_\infty(|E_{0}^\omega(X_n)-nv|^2)}=O(n^\alpha).
\end{align}
\begin{remarks}\label{rmk1}
The assumption (A1) implies that the particle has finite jump at
each transition. For instance, setting M=1, we get the nearest
neighbor random walk in random environment. The assumption (A2)
 shows that the variance of the quenched mean has a subdiffusive
 bound under the invariant and ergodic measure $\pp_\infty$.
\end{remarks}

Define the drift for the random walk in random environment as
follows
\begin{align}
D(\omega)=E_{0}^{\omega}X_{1}=\sum_{z\in\zz^d}z\pi_{0,z}(\omega).
\end{align}
Notice that the assumption (A1) yields $D\in L^2(\pp_\infty)$.
Denote $v=\ee_\infty D$ the drift under the annealed law
$P_{0}^\infty$. If set $g=D-v$, then
\begin{align}
X_{n}-nv&=X_{n}-\sum_{k=0}^{n-1}D(T_{X_{k}}\omega)+M^{\epsilon}_{n}+R^{\epsilon}_{n}+
\epsilon S_{n}(h_{\epsilon})\nonumber\\
&=W_n+M^{\epsilon}_{n}+R^{\epsilon}_{n}+ \epsilon
S_{n}(h_{\epsilon}).
\end{align}
where $W_n=X_{n}-\sum_{k=0}^{n-1}D(T_{X_{k}}\omega)$ is a martingale
under $P_{0}^{\omega}$ with respect to the filtration $\{{\cal
G}_n=\sigma(X_0,X_1,\cdots,X_n),\,n\geq0\}$ for $\pp_\infty$-a.s.
$\omega$.

Under the assumptions (A1) and (A2), Rassoul-Agha and
Sepp\"al\"ainen \cite{RSb} obtained the invariance principle for
random walks in random environments. We summarize their results in
the following theorems.
\vspace{3mm}\\
\textbf{Theorem RS} {\it Let $d\geq1$ and assume that (A1) and (A2)
are satisfied.
\vspace{3mm}\\
\noindent(1) The limit $H=\lim_{\epsilon\rightarrow0^+}H_\epsilon$
exists in $L^2(\nu^\infty)$.
\vspace{3mm}\\
\noindent(2) Denote
$M_n=\sum_{k=0}^{n-1}H(T_{X_{k}}\omega,T_{X_{k+1}}\omega)$, then
$X_n-nv=W_n+M_n+R_n$, $E_0^\infty(|R_n|^2)=O(n^{2\alpha})$, and for
$\pp_\infty$-almost surely $\omega$, $(M_n,\;n\geq1)$ is a
$P_0^\omega$-square integrable martingale relative to the filtration
$\{{\cal G}_n,\,n\geq0\}$.
\vspace{3mm}\\
\noindent(3) For $\pp_\infty-a.s. \,\omega$,
$n^{-1/2}(X_{[n\cdot]}-[n\cdot]v)$ converges in distribution to the
Brown motion with diffusion matrix $\mathfrak{D}$, under
$P^{\omega}_{0}$. Furthermore, $n^{-1/2}\max_{k\leq n}
|E^{\omega}_{0}X_k-kv|$ converges to zero, $\pp_\infty-a.s.
\,\omega$, and then the same invariance principle holds also for
$n^{-1/2}(X_{[n\cdot]}-E^{\omega}_{0}X_{[n\cdot]})$. Where
$v=\ee_\infty D$ is the drift under the annealed law $P_0^\infty$,
and $\mathfrak{D}=E_0^\infty[(X_1-D(\omega)+H(\omega,T_{X_1}\omega))
(X_1-D(\omega)+H(\omega,T_{X_1}\omega))^t]$ is the diffusion matrix
($A^t$ denotes the transpose of matrix or vector $A$ ).}

\begin{remarks}\label{rmk2}
Furthermore, we know that $H\in L^q(\nu^\infty)$ for some
$q\in(2,5/2)$ since the environment is finite (See Theorem 1 in
\cite{RSa}). For the more detailed discussions on the above theorem,
please see Rassoul-Agha and Sepp\"al\"ainen \cite{RSb}.
\end{remarks}

In order to obtain the invariance principle for the law of iterated
logarithm for random walks in random environments, we need the
additional assumption, \vskip10pt

\noindent (A3) {\it For any $\omega\in\Omega$, there exist an
integer $l\geq1$, $0<\lambda\leq1$ and a measure $\mu$ on
$(\Omega,\Im)$ such that}
\begin{align}
\sum_{x_1,x_2,\ldots,x_l\in E;|x_i|\leq M, 1\leq i\leq
l}\pi_{0x_1}(\omega)\cdots\pi_{0x_l}(T_{x_1+\ldots+x_{l-1}}\omega)
\textbf{1}_{A}(T_{x_1+\ldots+x_l}\omega)\geq\lambda\mu(A),\;A\in\Im.
\end{align}
\begin{remarks}\label{rmk3}
This assumption is a technical condition, since we need the
auxiliary Markov chain constructed above has the space $\Omega$ as
its a small set. It is the further task to explain and remove this
assumption.
\end{remarks}

For introduce our main results, we firstly give some notations. Let
$C([0,1],\rr^d)$ be the Banach space of continuous maps from $[0,1]$
to $\rr^d$, endowed with the supremum norm $\|\cdot\|$, using the
Euclidean norm in $\rr^d$. Denote $K$ the set of absolutely
continuous maps $f\in C([0,1],\rr^d)$, such that
\begin{align}
f(0)=0, \ \ \ \ \int_{0}^{1}|\dot{f}(t)|^2dt\leq1,
\end{align}
where $\dot{f}$ denotes the derivative of $f$ determined almost
everywhere with respect to Lebesgue measure. Obviously, $K$ is
relatively compact and closed.

Let $d\geq1, X=(X_{n},n\geq0)$ be a random walk in random
environment. Define for $t\in[0,1]$,
\begin{align*}
\xi_{n}(t)=(2v_{n}^2\log\log
v_{n}^2)^{-1/2}(X_k-kv+(X_{k+1}-X_{k}-v)(v_{k+1}^2-v_{k}^2)^{-1}(tv_{n}^2-v_{k}^2))
\end{align*}
for $v_{k}^2\leq tv_{n}^2\leq v_{k+1}^2, \;k=0,1,2,\ldots,n-1$,
where $v_{n}^2$ denotes the trace of the matrix given in Section
\ref{sec3}. In order to avoid difficulties in specification, we
adopt the convention that $\log\log x=1$, if $0<x\leq e^e$. Then,
$\xi_{n}$ is a random element with values in $C([0,1],\rr^d)$.
\vspace{3mm}

After these preparations, we are now in a position to state our main
results. \vspace{3mm}

\noindent\textbf{Theorem 1.} {\it Under the Assumptions (A1), (A2)
and (A3), for $P_{0}^{\infty}-a.s.$, the sequence of functions
$(\xi_{n}(\cdot), \,n\geq1)$ is relatively compact in the space
$C([0,1],\rr^d)$, and the set of its limit points as
$n\rightarrow\infty$, coincides with $K$.} \vspace{3mm}

\noindent\textbf{Theorem 2.} {\it If assumptions (A1), (A2) and (A3)
are satisfied, then
\begin{align}\label{thm2}
\limsup|X_{n}-nv|/\sqrt{2n\log\log
n}<+\infty,~~~~~~~~P_{0}^\infty-a.s.
\end{align}
Furthermore, we have
\begin{align}
\limsup|X_{n}-nv|/\sqrt{2n\log\log n}=\sqrt{\rm tr(\mathfrak{D})}
,~~~~~~~~P_{0}^\infty-a.s.
\end{align}
where {\rm tr}$(\cdot)$ denotes the trace operator of a matrix.}

\begin{remarks}\label{rmk4}
It is clear that the statement, almost surely relatively compact of
sequence $((X_{n}-nv)/\sqrt{2n\log\log n}, \;n\geq1)$ in $\rr^d$, is
equivalent to (\ref{thm2}).
\end{remarks}

\begin{remarks}\label{rmk5}
In generally, Theorem 1 is called Strassen's strong invariance
principle or functional LIL. Moreover, Theorem 1 and Theorem 2 also
hold, when the normalized center is random, $E^{\omega}_{0}X_{n}$,
by Theorem RS.
\end{remarks}

\section{The proof of main results}\label{sec3}

In this section, we will prove our main results, Theorem 1 and
Theorem 2, mentioned in Section \ref{sec2}, via the martingale
approach and the theory of fractional coboundaries.

Denote $Z_n=W_n+M_n$ and let $w_n,\,m_n$ and $z_n$ the martingale
difference corresponding to $W_n,\,M_{n}$ and $Z_{n}$ respectively.
It is easy to see the following facts,
\begin{align*}
(\clubsuit)\,&(w_{n},\,n\geq1) \textrm{\;is uniformly bounded
martingale
difference sequence under}\;P_{0}^{\omega}; \\
(\spadesuit)\,&(m_{n},\,n\geq1) \textrm{\;is stationary and ergodic
sequence under}\; P_0^\infty.
\end{align*}

Define for each $n$ the conditional covariance matrix
\begin{align}
A_{n}:=\sum_{k=1}^{n}E_0^{\omega}(z_{k}z_{k}^{t}|{\cal{G}}_{k-1}),
\end{align}
and set $v_{n}^2={\rm tr}(A_n)$. Note that,
$v_{n}^2=\sum_{k=1}^{n}E_0^{\omega}(|z_{k}|^2|\,{\cal{G}}_{k-1})$.
We know by the Markov property,
\begin{align}
A_{n}&=\sum_{k=1}^{n}E_0^{T_{X_{k-1}}\omega}(z_{1}z_{1}^{t})\nonumber\\
&=\sum_{k=1}^{n}E_0^{\bar{\omega}(k-1)}(z_{1}z_{1}^{t}),
\end{align}
where $\bar{\omega}=(\bar{\omega}(n)=T_{X_{n}}\omega,n\geq0)$ a
stationary and ergodic Markov chain under $P_0^\infty$ with initial
distribution $\pp_\infty$, since the discussions in Section
\ref{sec2}. Hence, from the Birkhoff and Khinchin's ergodic theory,
we know that,
\begin{align}
\lim_{n}n^{-1}A_{n}=\mathfrak{D},~~~~~~~~P_0^\infty-a.s.
\end{align}
And we also have,
\begin{align}\label{40}
\lim_{n}n^{-1}v_{n}^2={\rm{tr}}(\mathfrak{D}),~~~~~~~~P_0^\infty-a.s.
\end{align}
For any $d\times d$ matrix A, define the matrix norm,
\begin{align}
\|A\|_m:=\sup_{u\in\rr^d,|u|=1}|Au|.
\end{align}

\subsection{ Proof of Theorem 1}\label{sec3.1}

We will prove separately, in a succession of steps. \vspace{3mm}
\\
\noindent{\bf Step I:} Denote $B(\cdot)$ the Brownian Motion in
$\rr^d$ with mean 0 and diffusion matrix $\Sigma$. Define
\begin{align}
B_n(t)=(2n\log\log n)^{-1/2}B(nt)
\end{align}
for $t\in[0,1]$ and $n\geq3$. Then we have, by the Theorem 1 of
Strassen \cite{Str} the sequence $\{B_n(\cdot),\,n\geq3\}$ is almost
surely relatively compact in $C([0,1],\rr^d)$ and the set of its
limit points coincides with $\sqrt{\rm{tr}(\Sigma)}K$.
\vspace{3mm}\\
\noindent {\bf Step \textrm{II}:} In this step, we mainly consider
the almost sure approximation of the martingale $Z_{n}$ introduced
in the above section by a suitable $\rr^d$-valued Brown motion.
Since Theorem 1.2 of Zhang \cite{Zh}, under suitable conditions,
i.e., the following (B1) and (B2):
\begin{align*}
&(B1)
~~~~\sum_{n\geq1}E_0^\infty(|z_n|^2{\bf1}_{\{|z_n|^2\geq f(n)\}}/f(n))<\infty,\\
&(B2)~~~~\|A_n-n\mathfrak{D}\|_m=o(f(n)),\;\;\;\;\;\;P_0^\infty-a.s.
\end{align*}
we have,
\begin{align}\label{approximation}
|\sum_{n\geq1}z_n{\bf1}_{\{v_n^2\leq
t\}}-B(t)|=O(t^{1/2}(f(t)/t)^{1/50d}), ~~d<\infty,~~P_0^\infty-a.s.
\end{align}
where, $f(x)$ is non-decreasing and tends to $\infty$, along the
positive axis, $f(x)(\log x)^\varrho/x$ is non-increasing for some
$\varrho>50d$, and $f(x)/x^{\delta}$ is non-decreasing for some
$0<\delta<1$. Please notice the difference, conditions (B1) and (B2)
here, with the equations (1.12) and (1.13) in Zhang \cite{Zh}.
Hence, the main object turns to check the conditions (B1) and (B2).

Firstly, we consider the condition (B1). If set
$y_{n}:=z_{n}/\sqrt{f(n)}$, then $(y_n, n\geq1)$ is also a
martingale difference sequence under $P_{0}^{\omega}$. Then the
above problem turns to be
\begin{align}
\sum_{n\geq1}E_0^\infty(|y_{n}|^2{\bf1}_{\{|y_{n}|\geq1\}})<\infty.
\end{align}
It is enough to show that for some $\kappa>0$,
\begin{align}
E_0^\infty(|y_{n}|^2{\bf1}_{\{|y_{n}|\geq1\}})=O(n^{-(1+\kappa)}).
\end{align}
Note that $z_{n}=w_{n}+m_{n}$, then
\begin{align}
|y_{n}|^2{\bf1}_{\{|y_{n}|\geq1\}}
&=|w_{n}+m_{n}|^2{\bf1}_{\{|y_{n}|\geq1\}}/f(n)\nonumber\\
&\leq (M^2+2M|m_{n}|+|m_{n}|^2){\bf1}_{\{|y_{n}|\geq1\}}/f(n).
\end{align}
Hence, we need to deal with three terms,
\begin{align*}
&\textrm{I}_n:={\bf1}_{\{|y_{n}|\geq1\}}/f(n),\\
&\textrm{II}_n:=|m_{n}|{\bf1}_{\{|y_{n}|\geq1\}}/f(n),\\
&\textrm{III}_n:=|m_{n}|^2{\bf1}_{\{|y_{n}|\geq\}}/f(n).
\end{align*}
Since $(m_{n},n\geq1)$ is stationary under $P_0^\infty$, the key
estimation naturally is the probability of the event
$\{|y_{n}|\geq1\}$.

Now put $f(x)=x^{\gamma}$ for some $\gamma\in(2/q,1)$, hence there
exists a constant $\kappa\in(0,\gamma q/2-1)$ such that $\gamma
q/2\geq1+\kappa$. If we assume that
\begin{align}\label{4e}
P_0^\infty(|y_{n}|\geq1)=O((n)^{-q(1+\kappa-\gamma)/(q-2)}).
\end{align}
Then
\begin{align}
&E_0^\infty({\bf1}_{\{|y_{n}|\geq1\}}/f(n))\nonumber\\
&=n^{-\gamma}P_0^\infty(|y_{n}|\geq1)\nonumber\\
&=O(n^{-(q+q\kappa-2\gamma)/(q-2)})\nonumber\\
&=o(n^{-(1+\kappa)}).
\end{align}
Next we consider ($\textrm{II}_n$) and ($\textrm{III}_n$). \vskip5pt
\noindent($\textrm{II}_n$)
\begin{align}
&E_0^\infty(|m_{n}|{\bf1}_{\{|y_{n}|\geq1\}})/f(n)\nonumber\\
&\leq n^{-\gamma}[E_0^\infty(|m_{n}|^q)]^{1/q}[P_0^\infty(|y_{n}|\geq1)]^{1-1/q}\nonumber\\
&=n^{-\gamma}[E_0^\infty|m_{1}|^{q}]^{1/q}[P_0^\infty(|y_{n}|\geq1)]^{1-1/q}\nonumber\\
&=O(n^{-[(q-1)(1+\kappa)-\gamma]/(q-2)})\nonumber\\
&=o(n^{-(1+\kappa)}).
\end{align}
($\textrm{III}_n$)
\begin{align}
&E_0^\infty(|m_{n}|^2{\bf1}_{\{|y_{n}|\geq1\}})/f(n)\nonumber\\
&\leq n^{-\gamma}[E_0^\infty(|m_{n}|^q)]^{2/q}[P_0^\infty(|y_{n}|\geq1)]^{1-2/q}\nonumber\\
&=O(n^{-(1+\kappa)}).
\end{align}
Finally we analyze the equation (\ref{4e}). It is sufficiently to
consider the following two parts,
\begin{align}
P_0^\infty(|m_{n}|\geq 2^{-1}n^{\gamma/2})\;\;~ {\rm{and}}\;\;
~P_0^\infty(|w_{n}|\geq 2^{-1}n^{\gamma/2}).
\end{align}
By Markov inequality and the facts $(\clubsuit)$ and $(\spadesuit)$,
we get
\begin{align}
&P_0^\infty(|m_{n}|\geq2^{-1}n^{\gamma/2})\leq(2^{q}E_0^\infty|m_1|^q)n^{-\gamma
q/2},\\
&P_0^\infty(|w_{n}|\geq2^{-1}n^{\gamma/2})\leq(2M)^{q}n^{-\gamma
q/2}.
\end{align}
If setting $\kappa:=\gamma q/2-1>0$, we have
\begin{align}
P_0^\infty(|y_{n}|\geq1)=O(n^{-\gamma q/2})=
O((n)^{-q(1+\kappa-\gamma)/(q-2)}).
\end{align}
Then this completes the discussion of condition (B1).

As for condition (B2), we need to estimate the rate of convergence
of
\begin{align*}
\|A_n-n\mathfrak{D}\|_m.
\end{align*}
That is to say, we want to have the following order estimations,
\begin{align}
\|\frac{A_n}{n}-\mathfrak{D}\|_m=o(n^{\gamma-1}),\;\;\;P_0^\infty-a.s.
\end{align}
or
\begin{align}
\|\frac{v_n^2}{n}-{\rm
tr}(\mathfrak{D})\|_m=o(n^{\gamma-1}),\;\;\;P_0^\infty-a.s.
\end{align}
Denote $\phi(\omega):=E_0^{\omega}(|z_{1}|^{2})$, then the above
problem turns to be the problem of ergodic convergence rate for
additive functionals of stationary and ergodic Markov chain, i.e.,
the rate of
\begin{align*}
\frac{1}{n}\sum_{k=0}^{n-1}\phi(\bar{\omega}(k))\longrightarrow\int\phi(\omega)
d\pp_\infty={\rm tr}(\mathfrak{D}),\;\;\;P_0^\infty-a.s.
\end{align*}

The above problem can be rewritten as follows, \vskip5pt

{\it "Given ergodic Markov Chain $(Y_k,k\ge1)$ and function $f$ with
$f\in L^{q/2}$ and $\int fd\pi=0$ (where $\pi$ is invariant
distribution) under what condition do we have
\[
n^{-\gamma}\sum_{k=1}^n f(Y_k)\longrightarrow 0 \;\;\;almost\;
surely?"
\]}
To answer this problem, we need Chen's theorem $\cite{Cha}$.
\vspace{3mm}\\
\noindent\textbf{Theorem Chen} {\it Let $\{Y_n\}_{n\geq0}$ be an
ergodic Markov chain with state space $(H,\mathcal{H})$, $f$ a
measurable function from $H$ to some separable Banach space $B$, and
let $1\leq p<2$. Then the following two statements (1) and (2) are
equivalent: \vskip5pt (1) For some (all) small set $C$,
\begin{align}\label{int}
\int_{C}\pi(dx)E_x\max_{n\leq\tau_C}||\sum_{0}^{n-1}f(Y_k)||^p<\infty,
\end{align}
and
\begin{align}\label{conv}
n^{-1/p}\sum_{0}^{n-1}f(Y_k)\longrightarrow0\;\;\;in \;probability.
\end{align}

(2) The Marcinkiewicz-Zygmund's law of the large numbers holds,
i.e.,
\begin{align}
\lim_{n\rightarrow\infty}n^{-1/p}\sum_{0}^{n-1}f(Y_k)=0, \;\;\;a.s.
\end{align}
} Since the remarks followed this theorem in Chen \cite{Cha}, we
know that the equation (\ref{int}) and the condition
\[
\int f(x)\pi(dx)=0,
\]
imply the equation (\ref{conv}) when $B=\mathbb{R}$.  Hence we only
need to find the suitable conditions to describe $(\ref{int})$.

Define
\begin{align*}
&\mathcal{S}:=\{{\rm \;all \;small \;sets}\},\\
\mathcal{S}_\varphi:=\{&A\in\mathcal{S}:\;\int_{C}\pi(dx)E_x\max_{n\leq\tau_A}\psi(|\sum_{1}^{n}f(Y_k)|)\},
\end{align*}
where $\psi(x)=x^{1/\gamma}$.

Notice the dichotomy results obtained in two of Chen's works
\cite{Cha} and \cite{Chb}, it is easy to show the following
statement by applying Chen's idea,
$$
\mathcal{S}_\varphi=\varnothing \;\;{\rm or}\; \;\mathcal{S}.
$$
By the assumption (A3), for all $\omega\in\Omega$,
\begin{align*}
q^{(l)}(\omega,A)\geq\lambda\mu(A),\;\; A\in\Im,
\end{align*}
we have the space $\Omega$ is a small set. Hence, the equation
(\ref{int}) turns to be,
\begin{align}\label{kejixing}
E_0^\infty|f(\omega)|^{1/\gamma}<\infty.
\end{align}
However
\begin{align*}
&f(\omega)=E_0^{\omega}|z_1|^2-E_0^\infty|z_1|^2,\,f\in
L^{q/2},\\
(2&<q<5/2,\,1/2<4/5<2/q<\gamma<1),
\end{align*}
these yield the above equation (\ref{kejixing}). Hence we complete
the discussion on the condition (B2).
\vspace{3mm}\\
\noindent {\bf Step \textrm{III}:} For $t\in[0,1]$, we define the
$\rr^d$-valued functions,
\begin{align*}
&\eta_n(t):=\eta(tv_n^2),\\
\tilde{\eta}_n(t):={\small\sum}_{k=1}^{n}z_{k}{\bf1}_{\{v_k^2\leq
tv_n^2\}}&+{\small\sum}_{k=0}^{n-1}z_{k+1}(v_{k+1}^2-v_{k}^2)^{-1}(tv_{n}^2-v_{k}^2){\bf1}_{\{v_k^2\leq
tv_n^2\leq v_{k+1}^2\}},
\end{align*}
where $v_{0}^2=0$ and
$\eta(t)=\sum_{n\geq1}z_{n}{\bf1}_{\{{v_{n}^2}\leq t\}}. $

Then, since the {\bf Step II}, we need to show
\begin{align}
\sup_{t\in[0,1]}|\tilde{\eta}_n(t)-\eta_n(t)|=o((2v_n^2\log\log
v_n^2)^{1/2}).
\end{align}
In fact,
\begin{align}
&\sup_{t\in[0,1]}|\tilde{\eta}_n(t)-\eta_n(t)|\nonumber\\
&=\max_{0\leq k\leq n-1}\sup_{v_k^2\leq
tv_n^2\leq v_{k+1}^2}|z_{k+1}(v_{k+1}^2-v_{k}^2)^{-1}(tv_{n}^2-v_{k}^2)|\nonumber\\
&=\max_{0\leq k\leq n-1}|z_{k+1}|.
\end{align}
Hence, we only need prove the below estimation
\begin{align}
\max_{0\leq k\leq n-1}|z_{k+1}|=o((2v_n^2\log\log v_n^2)^{1/2}).
\end{align}
If we notice that
\begin{align}
z_{k}=w_{k}+m_{k},
\end{align}
and the fact $(\clubsuit)$, then the problem turns to be
\begin{align}\label{4f}
\max_{1\leq k\leq n}|m_{k}|=o((2v_n^2\log\log
v_n^2)^{1/2}),~~~~~~~~P_0^\infty-a.s.
\end{align}
It is easy to see, for any $\epsilon>0$,
\begin{align}\label{49}
&P_0^\infty(\max_{1\leq k\leq
n}|m_{k}|\geq\epsilon(2n\log\log n)^{1/2})\nonumber\\
&\leq E_0^\infty(\max_{1\leq k\leq
n}|m_{k}|^q)/\epsilon^{q}(2n\log\log n)^{q/2}.
\end{align}
Next, let us give the estimation of $E_0^\infty(\max_{1\leq k\leq
n}|m_{k}|^q)$. The following important inequality is a moment
inequality from M\'oricz \cite{Mor}.

\vspace{3mm} \noindent\textbf{Lemma M} {\it Let $p
> 0$ and $\beta
> 1$ be two positive real numbers and $Z_i$ be a sequence of random
variables. Assume that there are nonnegative constants $a_j$
satisfying
\begin{align}
E|\sum_{j=1}^iZ_j|^p\leq (\sum_{j=1}^ia_j)^\beta,
\end{align}
for $1\leq i\leq n$. Then
\begin{align}
E(\max_{1\leq i\leq n}|\sum_{j=1}^iZ_j|^p)\leq
C_{p,q}(\sum_{i=1}^na_i)^\beta,
\end{align}
for some positive constant $C_{p,\beta}$ depending only on $p$ and
$\beta$. }
\begin{lem}
For any enough large $n$, there exists a positive constant C such
that
\begin{align}
E_0^\infty(\max_{1\leq i\leq n}|m_i|^q)\leq CE_0^\infty|m_1|^q.
\end{align}
\end{lem}
\begin{proof}
Let $E_0^\infty|m_1|^q=a^2(q)$. Since the fact ($\spadesuit$), for
any $k\geq1$, we have the following relation,
\begin{align}
E_0^\infty|m_k|^q\leq (\sum_{i=1}^ka_i)^2,
\end{align}
where $a_{1}=a(q)$ and $a_{i}=0$ for $2\leq i\leq k$. Hence, by
Lemma M, there exists a constant $C>0$, such that
\begin{align}
E_0^\infty(\max_{1\leq i\leq n}|m_i|^q)\leq C(\sum_{i=1}^n
a_i)^2=CE_0^\infty(|m_1|^q).
\end{align}

This completes the proof of the lemma.
\end{proof}

Lemma 4.1 together with equation (\ref{49}) immediately yields,
\begin{align}
P_0^\infty(\max_{1\leq k\leq n}|m_{k}|\geq\epsilon(2n\log\log
n)^{1/2})=O((n\log\log n)^{-q/2}).
\end{align}
Hence, the above estimation (\ref{4f}) is obtained, by
Borel-Cantelli's lemma and equation (\ref{40}).
\vspace{3mm}\\
\noindent{\bf Step \textrm{IV}:} We want to give the order of
\begin{align}
\sup_{t\in[0,1]}|\tilde{\eta}_n(t)-B(tv_n^2)|.
\end{align}
Firstly, we rewrite it as follows,
\begin{align}
&\sup_{t\in[0,1]}|\tilde{\eta}_n(t)-B(tv_n^2)|\nonumber\\
&=\sup_{t\in[0,1]}|\tilde{\eta}_n(t)-\eta_n(t)+\eta_n(t)-B(tv_n^2)|\nonumber\\
&\leq\sup_{t\in[0,1]}|\tilde{\eta}_n(t)-\eta_n(t)|+\sup_{t\in[0,1]}|\eta_n(t)-B(tv_n^2)|.
\end{align}
From above {\bf Step III}, we know
\begin{align*}
\sup_{t\in[0,1]}|\tilde{\eta}_n(t)-\eta_n(t)|=o((2v_n^2\log\log
v_n^2)^{1/2}).
\end{align*}
However, by the equation (\ref{approximation}) in {\bf Step II}, we
have the following estimation,
\begin{align}
&\sup_{t\in[0,1]}|\eta_n(t)-B(tv_n^2)|\nonumber\\
&=\sup_{t\in[0,1]}O((tv_{n}^2)^{1/2}(f(tv_{n}^2)/tv_{n}^2)^{1/50d})\nonumber\\
&=o((2v_n^2\log\log v_n^2)^{1/2}).
\end{align}
This gives the following order estimation,
\begin{align}
\sup_{t\in[0,1]}|\tilde{\eta}_n(t)-B(tv_n^2)|=o((2v_n^2\log\log
v_n^2)^{1/2}),~~~~~~~~P_0^\infty-a.s.
\end{align}
\vspace{3mm}\\
\noindent{\bf Step\textrm{V}:} Notice that,
\begin{align}
&\sup_{t\in[0,1]}|\tilde{\eta}_{n}(t)-(2v_n^2\log\log
v_n^2)^{1/2}\xi_{n}(t)|\nonumber\\
&=\max_{0\leq k\leq n-1}\sup_{v_{k}^2\leq tv_{n}^2\leq
v_{k+1}^2}|R_{k}+(v_{k+1}^2-v_{k}^2)^{-1}(R_{k+1}-R_{k})(tv_{n}^2-v_{k}^2)|\nonumber\\
&\leq3\max_{1\leq k\leq n}|R_{k}|.
\end{align}
Hence, all things will boil down, if we show the following
estimation,
\begin{align}\label{4g}
\max_{1\leq k\leq n}|R_k|=o((2v_{n}^2\log\log v_{n}^2)^{1/2}),
~~~~~~~~P_0^\infty-a.s.
\end{align}
Define
\begin{align}
\varphi(\omega_0,\omega_1):=g(\omega_0)-H(\omega_0,\omega_1),
\end{align}
and by a simple calculation,
\begin{align}
    R_{n}&=\sum_{k=0}^{n-1}[g(T_{X_k}\omega)-H(T_{X_{k}}\omega,T_{X_{k+1}}\omega)]\nonumber\\
    &=\sum_{k=0}^{n-1}\varphi(T_{X_{k}}\omega,T_{X_{k+1}}\omega).
\end{align}
For a sequence
$\hat{\omega}=(\omega^{(i)})_{i\in\nn}\in\Omega^{\nn}$, define
$$
\Phi(\hat{\omega})=\varphi(\omega^{(0)},\omega^{(1)}) \;\;\;\;{\rm
and} \;\;\;\;\hat{R}=\sum_{k=0}^{n-1}\Phi\circ\theta^{k},
$$
where $\theta$ is the shift map on the sequence space $\Omega^{\nn}$
and is also a contraction on the space $L^2(\hat{P}_{0}^\infty)$.
Then $\Phi\in L^2(\hat{P}_{0}^\infty)$ and the process
$(\hat{R})_{n\geq1}$ has the same distribution under
$\hat{P}_{0}^\infty$ as the process $(R)_{n\geq1}$ has under
$P_0^\infty$.

The part (2) of Theorem RS tells $
E_0^\infty(|R_n|^2)=O(n^{2\alpha})$, then there exists a constant
$1/2<c_0<1-\alpha$, such that
\begin{align}
\sup_{n}||n^{c_{0}-1}\sum_{k=0}^{n-1}\Phi\circ\theta^{k}||<\infty.
\end{align}
 By the Theorem 2.17 of Derrienne and Lin \cite{DL} and $\alpha<1/2$, we have
$\Phi\in(I-\theta)^\eta L^2(\hat{P}_0^\infty)$, where
$\eta\in(1/2,1-\alpha)$. Using again Drrienne and Lin's Theorem 3.2
of \cite{DL}, we get
\begin{align}
|\hat{R}_{n}|=o((2n\log\log n)^{1/2}),\;\hat{P}_0^\infty-a.s..
\end{align}
Hence, $|R_{n}|=o((2n\log\log n)^{1/2}),\,P_0^\infty-a.s.$.
Moreover, applying an elementary property of real convergent
sequences, we immediately get
\[
\max_{1\leq k\leq n}|R_k|=o((2n\log\log n)^{1/2}),\,P_0^\infty-a.s..
\]
Together with the equation (\ref{40}), we prove the above equation
(\ref{4g}).

\subsection{Proof of Theorem 2}\label{sec3.2}
Here, we take along the lines of the proof of Theorem 4.8 in Hall
and Heyde \cite{HH}. For any $\rr^d$-valued function $f$, denote
$f=(f_{1},f_{2},\ldots,f_{d})^t$. By the definition of $K$, we have,
for any $f\in K$,
\begin{align}
|f(t)|^2&=\sum_{i=1}^{d}(\int_{0}^{t}\dot{f}_{i}(s)ds)^2\nonumber\\
&\leq\sum_{i=1}^{d}(\int_{0}^{t}\dot{f}_{i}(s)^2ds)\int_{0}^{t}1ds\leq
t,
\end{align}
where the first inequality by the Cauchy-Schwartz's inequality. So,
$|f(t)|\leq\sqrt{t}$. It follows that $\sup_{t\in[0,1]}|f(t)|\leq1$.
Hence, by Theorem 1,
\begin{align}
\limsup\sup_{t\in[0,1]}|\xi_{n}(t)|\leq1,~~~~~~~~P_0^\infty-a.s.
\end{align}
and, setting $t=1$, with the equation (\ref{40}) again,
\begin{align}
\limsup|X_{n}-nv|/\sqrt{2n\log\log n}\leq
\sqrt{\textrm{tr}(\mathfrak{D})},~~~~~~~~P_0^\infty-a.s.
\end{align}

On the other hand, we can put
$f(t)=\frac{t}{\sqrt{d}}\sum_{i=1}^{d}e_{i}$, $t\in [0,1]$. Then,
$f\in K$ and so for $P_0^\infty-a.s.\,\omega^{*}$, there exists a
sequence $n_{k}=n_{k}(\omega^{*})$, such that
\begin{align}
\xi_{n_{k}}(\cdot)(\omega^*)\longrightarrow f(\cdot).
\end{align}
Particularly, $f(1)=\frac{1}{\sqrt{d}}\sum_{i=1}^{d}e_{i}$,
$|\xi_{n_{k}}(1)(\omega^*)|\longrightarrow |f(1)|$. That is to say,
\begin{align}
|X_{n_{k}}-n_{k}v|/\sqrt{2{n_{k}}\log\log
n_{k}}\longrightarrow\sqrt{\textrm{tr}(\mathfrak{D})},~~~~~~~~P_0^\infty-a.s.
\end{align}
This completes the proof of Theorem 2.

\section*{\small Acknowledgements}
{\small The first author wishes to thank Prof. F.Q. Gao, Dr. Zh.H.
Du, and Dr. P. Lv for their kindly help, especially Professor X.
Chen for his valuable insight to the Step II in the proof of Theorem
1. The second author wishes to thank Prof. L.M. Wu of Universit\'e
Blaise Pascal and Wuhan University, for his helpful discussions and
suggestions during writing this paper.}

\end{document}